\documentclass[a4paper,11pt]{amsart}
\usepackage{a4wide,amsfonts,amsmath,amssymb,amsthm,epsfig,enumerate}
\usepackage{thm-restate}
\usepackage{xcolor}
\usepackage[colorlinks=true, allcolors=blue]{hyperref}
\usepackage[capitalise]{cleveref}

\graphicspath{{graphics/}}

\usepackage{enumitem}

\newtheorem{theorem}{Theorem}
\newtheorem{lemma}[theorem]{Lemma}
\newtheorem{corollary}[theorem]{Corollary}

\newtheorem{observation}[theorem]{Observation}
\newtheorem{definition}[theorem]{Definition}

\newcommand{\binnr}{\binom{[n]}{r}}
\newcommand{\binnd}{\binom{[n]}{d}}
\newcommand{\R}{\mathbb{R}}
\newcommand{\Z}{\mathbb{Z}}

\newcommand{\C}{\mathcal{C}}
\newcommand{\D}{\mathcal{D}}
\newcommand{\F}{\mathcal{F}}
\newcommand{\calS}{\mathcal{S}}

\newcommand{\tn}{\tilde{n}}
\newcommand{\tr}{\tilde{r}}

\definecolor{defblue}{rgb}{0.1,0.1,0.7}
\newcommand{\defi}[1]{\textcolor{defblue}{\emph{#1}}}

\title{Signotopes with few plus signs}

\author[H.~Bergold]{Helena Bergold}
\address[H.~Bergold]{Institut für Informatik, Freie Universit\"at Berlin, Germany}
\email{}

\author[L.~Egeling]{Lukas Egeling}
\address[L.~Egeling]{Department of Computer Science, ETH Zurich, Switzerland}
\email{}

\author[H.~P.~Hoang]{Hung. P. Hoang}
\address[H.~P.~Hoang]{Algorithm and Complexity Group, Faculty of Informatics, TU Wien, Austria}
\email{phoang@ac.tuwien.ac.at}

\thanks{Hung P. Hoang acknowledges support from the Austrian Science Foundation (FWF, project Y1329 START-Programm).
We would like to thank Bernd G\"{a}rtner for stimulating discussions on the paper.}

\begin{document}

\maketitle

\begin{abstract}
    Arrangements of pseudohyperplanes are widely studied in computational geometry. 
    A rich subclass of pseudohyerplane arrangements, which has gained more attention in recent years, is the so-called signotopes. 
    Introduced by Manin and Schechtman (1989), the higher Bruhat order is a natural order of $r$-signotopes on $n$ elements, with the signotope corresponding to the cyclic arrangement as the minimal element.  
    In this paper, we show that the lower (and by symmetry upper) levels of this higher Bruhat order contain the same number of elements for a fixed difference $n-r$.
    This result implies that given the difference $d=n-r$ and $p$, the number of one-element extensions of the cyclic arrangement of $n$ hyperplanes in $\R^d$ with at most $p$ points on one side of the extending pseudohyperplane does not depend on $n$, as long as $n \geq d + p$. 
\end{abstract}

\section{Introduction}

Arrangements of geometrical and topological objects play an important role in computational geometry. 
A classical example is line arrangements, which consist of a family of pairwise intersecting lines. 
For various occasions, it is useful to study the combinatorial structure instead of the actual geometry. 
For this reason, pseudoline arrangements have gotten more and more attention after Levi's introductory paper~\cite{Levi1926}. 
In contrast to the geometric setting, it is possible to encode the structure of pseudolines efficiently.
To determine the size of an optimal encoding for point sets, Knuth~\cite{Knuth1992} studied the number of the so-called CC-Systems. 
By duality, there are the same number of pseudoline arrangements, which is $2^{\Theta(n^2)}$~\cite{Knuth1992,FelsnerGoodman2017}.
However, the precise leading constant in the exponent is still unknown, which leaves Knuth's question about the optimal encoding size open~\cite{Knuth1992}. 
Recently, Cort\'es Kühnast, Dallant, Felsner, and Scheucher~\cite{KuhnastDFS24} determined the best-known lower bound of $2^{0.2721n^2} $ pseudoline arrangements.
To tackle this challenging task, they used 3-signotopes for an encoding. Since the order on the elements itself is fixed, the set of 3-signotopes has less symmetric copies, which makes the encoding easier. 
As shown by Felsner and Weil~\cite{FelsnerWeil2001} there is a bijection between pseudoline arrangements and 3-signotopes. 

In this paper, we consider higher dimensions. 
Pseudohyperplane arrangements in $\R^d$ are combinatorially encoded as uniform acyclic oriented matroids of rank $r= d+1$. 
For more about oriented matroids, see~\cite{BjoernerLVWSZ1993}.
The asymptotic number of oriented matroids of rank $r$ is $2^{\Theta(n^{r-1})}$. 
However, little is known about the constant in the exponent. 
Instead of the general case, we consider signotopes of rank $r$, a rich subclass of pseudohyperplane arrangements, of which there are also $2^{\Theta(n^{r-1})}$, see~\cite{Balko19, BergoldFS23}.  
The leading constant is also unknown.
In this paper, we determine the number of signotopes with few plus (or symmetrically, few minus) signs.
This may be a step towards determining the precise number of all signotopes in the future. 

To formally define signotopes, let $[a,b]$ denote the set $\{x \in \Z \mid x \geq a \land x \leq b\}$ for two integers $a$ and $b$.
Moreover, let $[n]$ be $[1,n]$ and ${[a,b] \choose r}$ denote the set of all subsets of size $r$ of $[a,b]$.
We call such a subset an \defi{$r$-subset} of $[a,b]$.

\begin{definition}[Signotope]
    For $r \geq 1$, a \defi{signotope of rank $r$} (or an \defi{$r$-signotope}) on $n$ elements is a sign function $\sigma: \binnr \to \{+,-\}$, such that for every $(r+1)$-subset $X = \{x_1, \dots, x_{r+1}\}$ of $[n]$ with $x_1 < \dots < x_{r+1}$, there is at most one sign change in the sequence 
    \[
        \Big(\sigma(X\setminus\{x_1\}), \dots, \sigma(X\setminus\{x_{r+1}\})\Big).
    \]
    The set of all $r$-signotopes on $n$ elements is denoted by $\calS(n,r)$.
\end{definition}

For fixed $n$ and $r$, we can define an order on the set $\calS(n,r)$ of all $r$-signotopes on $n$ elements based on the number of $+$-signs. 
Specifically, two sign functions $\sigma$ and $\tilde{\sigma}$ differ in a \defi{single step}, if $\sigma^{-1}(+) \subset \tilde{\sigma}^{-1}(+)$ and $|\tilde{\sigma}^{-1}(+)| = |\sigma^{-1}(+)| + 1$.
The transitive closure of the single step relation gives the higher Bruhat order. 
\begin{definition}[Higher Bruhat Order]
    For $r \geq 1$, the \defi{higher Bruhat order} $B(n,r)$ is the partial order $\preceq$ on the elements of $\calS(n,r)$ with the transitive closure of the single step.
\end{definition}
Signotopes are the elements of the higher Bruhat order, which were introduced by Manin and Schechtman~\cite{manin1989arrangements} as a higher dimensional equivalence of the weak Bruhat order for permutations. 
They were further studied by Kapranov and Voevodsky~\cite{KaVo91}, Ziegler~\cite{Ziegler1993}, and later Felsner and Weil~\cite{FelsnerWeil2000}. 
In these works they give several geometric interpretations. 
The most important one for our purposes is the representation of an $r$-signotope on $[n]$ as a one-element extension of the cyclic arrangement of $n$ hyperplanes in $\R^{n-r}$~\cite{Ziegler1993}. 

\begin{definition}[Cyclic arrangement]
    The \defi{cyclic arrangement}~$\mathbf{X}^{n,d}_c$ is the set $\{H_1, \dots, H_n \}$ in $\R^d$ given by
    \[
        H_i = \{ (x_1, \dots, x_d) \in \R^d: x_1 + t_i x_2 + \dots + t_i^{d-1} x_d + t_i^d = 0 \}
    \]
    for $i \in [d]$, with arbitrary real parameters $t_1 < \dots < t_n$.
\end{definition}

A \defi{point} of the cyclic arrangement~$\mathbf{X}^{n,n-r}_c$ is the intersection of $n-r$ hyperplanes and can be identified with the $n-(n-r) = r$ hyperplanes which do not contain this point. 
In~$\mathbf{X}^{n,n-r}_c$, along each \defi{line} (i.e., the intersection of $n-r-1$ hyperplanes), the corresponding $r$-subsets of the crossing points are ordered lexicographically. 
A \defi{one-element extension} of $\mathbf{X}^{n,n-r}_c$ is obtained by adding a pseudohyperplane to $\mathbf{X}^{n,n-r}_c$ in general position (i.e., the pseudohyperplane does not contain any point of $\mathbf{X}^{n,n-r}_c$). 
Additionally, the \defi{extending pseudohyperplane} (i.e., the new pseudophyperplane) is oriented; that is, it divides $\R^d$ into two sides, a positive and a negative side. 
If a point lies in the positive side, the corresponding $r$-subset is mapped to $+$, and that correpsonding to a point in the negative side is mapped to $-$. 
Monotonicity of this sign mapping on $r$-subsets follows from the lexicographical order of the $r$-subsets along the lines. 
For an illustration see \cref{fig:1elementextension}.

\begin{figure}[htb]
    \centering
    \includegraphics[width=1\linewidth]{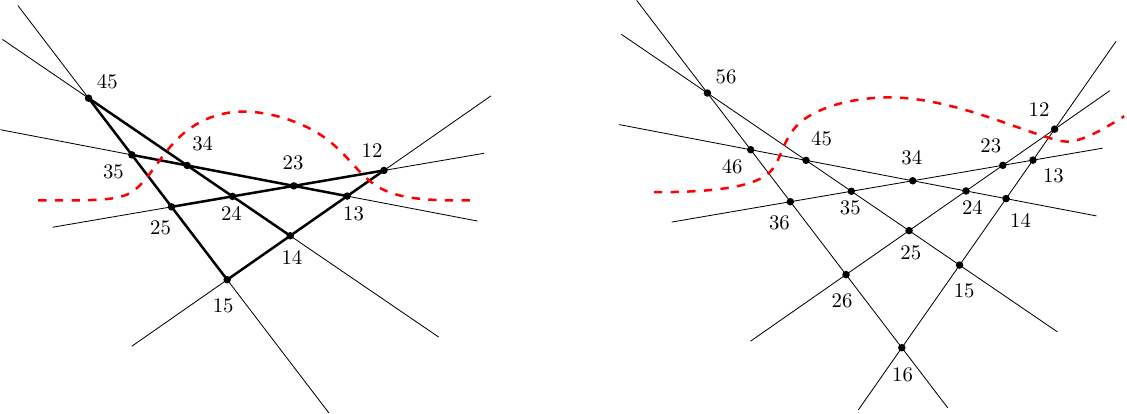}
    \caption{Two examples of one-element extensions of the cyclic arrangements~$\mathbf{X}^{5,2}_c$ (left) and~$\mathbf{X}^{6,2}_c$ (right). The red dashed lines are the extending pseudolines. To avoid clusters, we label each point of the arrangement by the lines that contain it (e.g., point 12 is the intersection of lines $H_1$ and $H_2$). The positive sides are the sides above the extending pseudolines. The corresponding two signotopes are mapped to each other in our bijection for \cref{thm:bijection}.
    }
    \label{fig:1elementextension}
\end{figure}

In this paper, we count the number of signtopes in a given level of the higher Bruhat order (i.e., the number of signotopes with a given number of $+$-signs).
We show that the bottom (and by symmetry the top) levels in the Bruhat order are the same whenever $n-r$ is the same. 
More precisely, we denote by $\calS_{\leq p}(n,r)$ the set of all $r$-signotopes~$\sigma$ on $n$ elements with at most $p$ $+$-signs (i.e., $|\sigma^{-1}(+)| \leq p$).
Our main result is then stated as follows.
\begin{restatable}{thm}{bijection}
\label{thm:bijection}
    Let $n, \tilde{n}, r, \tilde{r},d, p$ be natural numbers such that $n-r = \tilde{n}-\tilde{r} = d$ and $p \leq \min\{r,\tr\}$.
    Then there exists a bijection between $\calS_{\leq p}(n,r)$ and $\calS_{\leq p}(\tilde{n},\tilde{r})$ that preserves the relation~$\preceq$.
\end{restatable}

In other words the theorem states that the lower levels of the higher Bruhat order contain the same number of elements for a fixed difference $n-r$.
Note that by symmetry the same holds for the upper levels, i.e., if there are $\leq r$ minus signs. 
Using the representation of signotopes as one-element extensions of the cyclic arrangement, the geometric interpretation of the theorem above is as follows.

\begin{corollary}
    Given $d,p \geq 0$, for any $n \geq d+p$, the number of one-element extensions of the cyclic arrangement~$\mathbf{X}^{n,d}_c$ with at most $p$ points on one side of the extending pseudohyperplane is a constant (which depends only on $d$ and $p$).
\end{corollary}

The representation also allows us to prove the theorem.
It turns out that whenever there are $\leq r$ plus signs (respectively minus signs), we only need to consider the local structure of the one-element extension. 
In particular, this local structure only depends on $p$ and $n-r$, and not the exact values of $n$ and $r$ themselves; this insight gives rise to the bijection in \cref{thm:bijection}.
In \cref{sec:proof_outline} we give a more detailed outline and in Sections~\ref{sec:plus_subsets}--\ref{sec:main_proof} we give the technical proof.
In order to show this dependency, we use a generalization of the Ferrers diagrams.
These diagrams are geometric illustrations of integer partitions~\cite{MR2122332}, which are classical objects in number theory and combinatorics.
This connection allows us not only to prove the bijection but also to quantify the size of $\calS_{\leq p}(n,r)$ in \cref{sec:count}.

\section{Proof overview for Theorem \ref{thm:bijection}}
\label{sec:proof_outline}
In this paper, since we discuss $r$-signotopes on $[n]$ where $n-r$ is fixed, it is more convenient to work with the complements of the $r$-subsets.
Hence in \cref{subsec:co}, we introduce the notion of co-signotopes with related definitions analogous to the signotopes and rephrase \cref{thm:bijection} in terms of this new notion.
Afterwards, we provide an overview of the proof.

\subsection{Co-signotopes}
\label{subsec:co}
For convenience, we write each subset as a sequence of its elements in the increasing order; that is we consider an $r$-subset $\{a_1, \dots, a_r\}$ with $a_1 < \dots < a_r$ as the tuple $(a_1, \dots, a_r)$.
This way, we can order the subsets in the lexicographical order of the tuple representations.

\begin{itemize}
    \item For a function $f$, a $d$-subset $B = (b_1, \dots, b_d)$ of $[n]$, and an index $i \in [d]$, let $B_i := B \setminus \{b_i\}$.
    Then the \defi{$(f,B,i)$-series} is the sequence of $f(B_i \cup \{x\})$ for $x \in [n] \setminus B_i$ in the increasing order.
    In other words, if $[n] \setminus B_i = \{x_1, \dots, x_{n-d+1}\}$ with $x_1 < \dots < x_{n-d+1}$, then the $(f,B,i)$-series is exactly the sequence
    \[
        (f(B_i\cup\{x_1\}), \dots, f(B_i\cup\{x_{n-d+1}\})).
    \]
    When $f$ is the identity function, then this corresponds to the sequence $(B_i\cup\{x_1\}, \dots, B_i\cup\{x_{n-d+1}\})$, and we call it simply the $(B,i)$-series.
    Note that the $(B,i)$-series contains $B$.

    \item For $0\leq d < n$, a \defi{co-signotope of rank $d$} (or a \defi{$d$-co-signotope}) on $n$ elements is a sign function $\tau: \binnd \to \{+,-\}$, such that for every $d$-subset $B$ of $[n]$ and index $i \in [d]$, there is at most one sign change in the $(\tau,B,i)$-series.
    \item The set of all $d$-co-signotopes on $n$ elements is denoted by $\bar{\calS}(n,d)$.
    \item Given a sign function $\tau$, a subset~$R$ is a \defi{$+$-subset} of $\tau$ if $\tau(R)=+$, and it is a \defi{$-$-subset} of $\tau$ if $\tau(R)=-$.
    The set of all $d$-co-signotopes on $n$ elements with exactly $p$ $+$-subsets is denoted by $\bar{\calS}_{p}(n,d)$. 
    The set of all $d$-co-signotopes on $n$ elements with at most $p$ $+$-subsets is denoted by $\bar{\calS}_{\leq p}(n,d)$. 
    
    \item The \defi{complementary higher Bruhat order} $\bar{B}(n,d)$ is the partial order $\preceq_{\bar{\calS}}$ on the elements of $\bar{\calS}(n,d)$ with the transitive closure of the single step.
\end{itemize}

The following observation follows immediately from the definitions of signotopes and co-signotopes.

\begin{observation}
    For non-negative integers $n, d, p$ such that $d \leq n$, there is an isomorphism between $S_p(n,n-d)$ and $\bar{S}_p(n,d)$ that preserves the single step relation.
\end{observation}

Then we can rephrase \cref{thm:bijection} in terms of co-signotopes as follows.

\begingroup
\def\thetheorem{\ref{thm:bijection}}
\begin{theorem}[Rephrased]
    Let $n, \tilde{n}, d, p$ be integers such that $0\leq p \leq \min\{n,\tn\}-d$.
    Then there exists a bijection between $\bar{\calS}_{\leq p}(n,d)$ and $\bar{\calS}_{\leq p}(\tn,d)$ that preserves the relation~$\preceq_{\bar{S}}$.
\end{theorem}
\addtocounter{theorem}{-1}
\endgroup

\subsection{Proof strategy of Theorem \ref{thm:bijection}}
\label{subsec:proof_outline}
Let $G_{n,d}$ be the underlying graph of the cyclic arrangement~$\mathbf{X}^{n,d}_c$; that is, its vertices are the points of~$\mathbf{X}^{n,d}_c$, and its edges are the segments of the lines of~$\mathbf{X}^{n,d}_c$ such that the segments have two points of~$\mathbf{X}^{n,d}_c$ as endpoints and contain no other points in the interior.
In other words, the vertices of $G_{n,d}$ are the $d$-subsets, and two $d$-subsets~$B$ and $B'$ are connected by an edge, if they are adjacent in the $(B,i)$-series for some $i \in [d]$.
For illustration, the bold line segments in \cref{fig:1elementextension}(left) form the edges of the graph $G_{n,d}$.

The proof of \cref{thm:bijection} hinges on the following main idea.
For a co-signotope~$\tau \in \bar{\calS}_{\leq p}(n,d)$, we consider the induced subgraph of $G_{n,d}$ on the $+$-subsets of $\tau$.
Then each connected component in this induced subgraph (which we also call the $+$-component of $\tau$) can be characterized by a structure around a "source" subset (defined in the next section).
In particular, locally, the $+$-component appears like a generalized version of the Ferrers diagrams.
Further, this characterization is independent from that of other $+$-components and depends only on the source subset, $d$, and $p$.
With this observation, we can independently map each $+$-component of $\tau$ to that of a co-signotope in $\bar{\calS}_{\leq p}(\tn,d)$.

In order to rigorously execute the idea above, we first provide some preliminary structural insights on the $+$-subsets of a co-signotope $\tau \in \bar{\calS}_{\leq p}(n,d)$ in \cref{sec:plus_subsets}.
Next, in \cref{sec:onepluscomponent}, we show how we can map one $+$-component of a co-signotope $\tau \in \bar{\calS}_{\leq p}(n,d)$ to that of a co-signotope $\tilde{\tau} \in \bar{\calS}_{\leq p}(\tn,d)$.
To ease the explanation, in this section, we assume that there is exactly one $+$-component for $\tau$ and for $\tilde{\tau}$.
Here is also where we explain the connection with the Ferrers diagrams.
Then in \cref{sec:decompositionPlusComp}, we show how we can map one signotope in $\bar{\calS}_{\leq p}(n,d)$ to a signotope in $\bar{\calS}_{\leq p}(\tn,d)$ based on their sets of $+$-components. 
Finally, \cref{sec:main_proof} wraps up the proof and provides a simple bijection as required by \cref{thm:bijection}.

\section{Structure of $+$-subsets of a co-signotope in $\bar{\calS}_p(n,d)$}
\label{sec:plus_subsets}
Consider a co-signotope in $\bar{\calS}_p(n,d)$ that has only one $+$-subset.
This unique $+$-subset needs to be a subset which appears at the start or end of every sequence it is contained in. 
The only $d$-subsets which fulfill this requirement are $S_{n,d,i}$ where $S_{n,d,i} = (1, \ldots, i, n-d+i+1, \ldots, n)$.
We call such an $S_{n,d,i}$ a \defi{source subset}.
For convenience, we drop the subscripts $n$ and $d$, when they are clear.

Recall that $G_{n,d}$ is the underlying graph of the cyclic arrangement~$\mathbf{X}^{n,d}_c$, as defined in \cref{subsec:proof_outline}.
We have the following lemma on the distance from a subset to a source subset.

\begin{lemma}
\label{lem:dist_to_source}
	Let $n,d,i$ be integers such that $1 \leq i \leq d \leq n$.
	Let $B = (b_1, \dots, b_d)$ be a $d$-subset.
	Then every path between $B$ and $S_{n,d,i}$ in $G_{n,d}$ has length at least $b_i-b_{i+1}+n-d+1$, where we define $b_0 := 0$ and $b_{d+1} := n+1$.
\end{lemma}

\begin{proof}
    We represent each $d$-subset as a set of $d$ tokens, one placed on the real line at each element of the subset.
    Then the act of going from one $d$-subset to an adjacent $d$-subset in $G_{n,d}$ can be seen as moving a token along the real line (in an arbitrary direction) to the next free integer. 
	Suppose that $b_i > i$.
	Note that $[b_i-1]$ contains at least $i$ elements of $S_i$ and exactly $i-1$ elements of $B$.
	Hence, in the process of going along a path from $S_i$ to $B$ in $G_{n,d}$, there must be at least one token from the interval $[b_i-1]$ going to the interval $[b_i,n]$. 
	Therefore, every integer in $[i+1,b_i]$ must be visited by a token at some point in the process (these integers may or may not be visited by the same token).
	Note that this statement is vacuously true when $b_i = i$ or when $i=0$.
	
	Using a symmetric argument, we also deduce that every integer in $[b_{i+1},n-d+i]$ must also be visited by a token at some point in the process above.
	Further, at any step, we only move one token, and hence, at most one integer in $[i+1,b_i] \cup [b_{i+1},n-d+i]$ is visited at any step.
	Hence, we conclude that the path from $S_i$ to $B$ must have length at least $(b_i-i) + (n-d+i-b_{i+1}+1)$.
	The lemma then follows.
\end{proof}

For a $d$-co-signotope $\tau$, let $G_{n,d}[\tau]$ be the induced subgraph of $G_{n,d}$ on $\tau^{-1}(+)$ (i.e., the $+$-subsets of $\tau$).
A maximal connected induced subgraph of $G_{n,d}[\tau]$ is called a \defi{$+$-component} of $\tau$.
Then it is easy to see that each $+$-component of $\tau$ contains at least one source subset.
This is actually tight for all components when $\tau \in \bar{\calS}_p(n,d)$ with $p  \leq n-d$, as stated in the following lemma.

\begin{lemma}
\label{lem:exact_one_source}
    Let $n,d,p$ be non-negative integers such that $p \leq n-d$.
    Then for $\tau \in \bar{\calS}_p(n,d)$, each $+$-component of $\tau$ has exactly one source subset.
\end{lemma}
\begin{proof}
    Suppose we have two source subsets $S_i$ and $S_j$ with $i<j$. 
    Applying \cref{lem:dist_to_source} with $B = S_j$, a path between two source subsets $S_i$ and $S_j$ with $i<j$ has length at least $i - (i+1) + n - d +1 = n-d$. This shows that every $+$-component containing $S_i$ and $S_j$ has at least $n-d+1 > p$ $+$-subsets. 
    Since $\tau$ has $p$ $+$-subsets, it follows that no component of $G_{n,d}[\tau]$ can contain more than one source subset.
    Combined with the fact that each component needs to have at least one source subset, the lemma then follows.
\end{proof}

Motivated by the lemma above, for a co-signotope $\tau \in \bar{\calS}(n,d)$ and for $i\in[0,d]$, we denote by $\C^{\tau}_i$ the $+$-component that contains $S_i$
($\C^{\tau}_i$ may be empty.)

For a co-signotope $\tau \in \bar{\calS}(n,d)$, a $d$-subset $B$, and an index $j \in [d]$, we say the $(\tau,B,j)$-series is \defi{left-aligned} if it starts with a $+$-subset and ends with a $-$-subset, \defi{right-aligned} if it starts with a $-$-subset and ends with a $+$-subset, and \defi{flat} if there is no sign change.
Since every $(\tau,B,j)$-series has $n-d+1$ subsets, if $\tau$ has at most $(n-d)$ $+$-subsets, then no flat series has only $+$-subsets.
Hence, we have the following simple observation.

\begin{observation}
\label{ob:no_flat}
    Let $n,d,p$ be non-negative integers such that $p \leq n-d$.
    Then for every $+$-subset $B$ of a co-signotope $\tau \in \bar{\calS}_p(n,d)$ and for every index $j \in [d]$, the $(\tau,B,j)$-series is not flat.
\end{observation}

The following lemma characterizes which $d$-subsets can be in the same $+$-component as some source subset $S_i$.

\begin{lemma}
\label{lem:one_source_index}
    Let $n,d,c,p,i$ be integers such that $1 \leq c \leq p \leq n-d$ and $0 \leq i \leq d$.
    Further, let $\tau$ be a co-signotope in $\bar{\calS}_{p}(n,d)$.
    Then if the $+$-component $\C^{\tau}_i$ has size $c$ and contains a $+$-subset $B = (b_1, \dots, b_d)$, then we have
    \begin{enumerate}[label=(\alph*)]
        \item For all $j \in [i]$, $b_j \leq c+j-1$ and the $(\tau,B,j)$-series is left-aligned; and
        \item For all $j \in [i+1,d]$, $b_j \geq n-(c+d-j)+1$ and the $(\tau,B,j)$-series is right-aligned.
    \end{enumerate}
\end{lemma}
\begin{proof}
    We first consider the case when $j \in [i]$.
    By \cref{lem:dist_to_source}, the shortest path between $B$ and $S_i$ is at least $b_i - b_{i+1} + n - d+1$.
    Since $B$ and $S_i$ are in the same $+$-component $\C^{\tau}_i$ of size $c$, this implies $b_i - b_{i+1} + n - d+1 \leq c-1$.
    Moreover since $b_{i+1} < b_{i+2} < \dots < b_d \leq n$, we obtain $b_{i+1} \leq n-d+i+1$.
    Together this implies
    \[
    	b_i \leq (c-1) - (- b_{i+1} + n - d+1) = c+i-1 + b_{i+1} - (n-d+i+1) \leq c+i-1.
    \]
    Combining this with the fact that $b_j \leq b_i - (i-j)$, we have $b_j \leq c+j-1$ as required.

    We now prove by contradiction that the $(\tau,B,j)$-series is left-aligned.
    Suppose this is not the case.
    Combined with \cref{ob:no_flat}, $(\tau,B,j)$-series must then be right-aligned, i.e., all of the subsets after $B$ in the $(B,j)$-series then have to be $+$-subsets.
    That means each element in $[b_j,n]$ appears in some $+$-subset.
    If $b_j = j$, then also all elements in $[b_j]$ appear in $S_i$, a $+$-subset.
    Otherwise, using the same token argument as in the proof of \cref{lem:dist_to_source}, each element in $[i+1,b_i]$ is contained in some $+$-subset along the path from $S_i$ to $B$ in $\C^{\tau}_i$.
    Combining with the fact that $[i] \subseteq S_i$, we obtain that each element in $[b_i]$ appears in some $+$-subset.
    Overall, each element in $[n]$ appears in some $+$-subset.

    Next, observe that going from a subset to an adjacent subset in $G_{n,d}$, we remove one element and add one element.
    Hence, since $\C^{\tau}_i$ has at most $p$ elements, the number of elements appearing in at least one $+$-subset is at most $d+p-1 \leq n-1$.
    This contradicts with the conclusion above that each element in $[n]$ appears in some $+$-subset.

    Hence, (a) holds.
    With a symmetric argument, we also obtain (b), completing the proof of the lemma.
\end{proof}

A direct consequence of the lemma above is that for a given $\tau \in \bar{\calS}_{p}(n,d)$, we can identify to which $+$-component a $+$-subset belongs. 

\begin{corollary}
\label{lem:which_component}
    Let $n,d,p$ be integers such that $1 \leq p \leq n-d$.
    Further, let $\tau$ be a co-signotope in $\bar{\calS}_{p}(n,d)$ and $B = (b_1, \dots, b_d)$ be a $+$-subset of $\tau$.
    Then $B$ is contained in the $+$-component $\C^{\tau}_i$, where $i$ is either the largest index in $[d]$ such that the $(\tau,B,i)$-series is left-aligned or zero if such an index does not exist.
\end{corollary}

\section{Co-signotopes with only one $+$-component}
\label{sec:onepluscomponent}

As a building block for the bijection, in this section, we consider $\tau \in \bar{\calS}_p(n,d)$ such that $n \geq d+p$ and $\tau$ has exactly one $+$-component.
We give a characterization of such a co-signotope via a variant of high-dimensional Ferrers diagrams.
For convenience, we denote by $\bar{\calS}_{p,i}(n,d)$ the set of all co-signotopes $\tau \in \bar{\calS}_p(n,d)$ such that $\C^{\tau}_i$ is the only non-empty $+$-component of $\tau$.

To facilitate the proof of the bijection we define a function $g_{n,d,i}$ for integers $n,d,i$ such that $0 \leq i \leq d < n$. This function allows us to consider the $+$-component locally with new coordinates, where $S_i$ has coordinate $(1,\ldots, 1)$, and to allude to the Ferrers diagrams.
In the first step, we consider the following auxiliary function:
\[
	f_{n,d,i,j}(x) =
	\begin{cases}
		x - j + 1& \text{if } 1 \leq j \leq i \\
		(n - x) - (d - j) + 1 & \text{if } i < j \leq d.
	\end{cases}
\]

Then we can now define $g_{n,d,i}$ that maps a $d$-subset to a point in the $d$-dimenional space.
\begin{definition}
\label{def:g}
    Let $n,d,i$ be integers such that $0 \leq i \leq d < n$.
    Then the function $g_{n,d,i}$ maps each $d$-subset $B = (b_1, \dots, b_d)$ to the tuple 
    \[
        \big(f_{n,d,i,1}(b_1), f_{n,d,i,2}(b_2) \dots, f_{n,d,i,d}(b_d)\big).
    \]
\end{definition}

We have the following observation on this function $g_{n,d,i}$.
\begin{observation} 
\label{ob:one_source_g}
    The function $f_{n,d,i,j}$ is injective and has an inverse, and hence, the same holds for $g_{n,d,i}$.
    Further, if $B$ is a $+$-subset of a co-signotope $\tau \in \bar{\calS}_p(n,d)$, then by \cref{lem:one_source_index} and the fact that $b_j \in [j,n-(d-j)]$, each index of $g_{n,d,i}(B)$ is in $[p]$.
    In particular, $g_{n,d,i}(S_i) = (1, \dots, 1)$.
\end{observation}

As mentioned before, $g_{n,d,i}$ allows us to make a connection with the Ferrers diagram of an integer partition~\cite{MR2122332}.
An \defi{integer partition} of a non-negative integer $n$ is a way of writing $n$ as a sum of positive integers, where the summands are sorted in the nonincreasing order.
The \defi{graphical representation} or \defi{Ferrers diagram} of an integer partition of $n$ is a set of points with positive integral indices in the plane such that $(i,j)$ is in the set only if $(i',j')$ is also in the set for all $i' \leq i$, $j' \leq i$.
It is easy to see the one-to-one correspondence between an integer partition and its graphical representation.
See \cref{fig:ferrers}(left) for an illustration.

\begin{figure}
    \centering
    \includegraphics{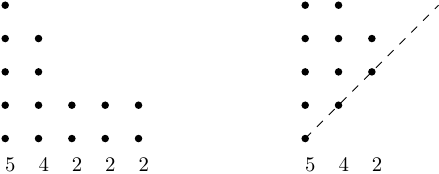}
    \caption{\textbf{Left}: A Ferrers diagram with the corresponding integer partition of 15, where the bottom-leftmost point corresponds to (1,1). 
    Note that this is also a $(2,1)$-Ferrers diagram with respect to 15.
    \textbf{Right}: A $(2,2)$-Ferrers diagram with respect to 11. 
    Since $a_2 \geq a_1$, there cannot be any point below the dashed line. The numbers are the count of points in each vertical line. 
    Observe that these numbers must be strictly decreasing.}
    \label{fig:ferrers}
\end{figure}

We define the following generalization of the Ferrers diagram.
\begin{definition}[$(d,i)$-Ferrers diagram]
\label{def:gen_ferrers}
    Let $Z(d,i)$ be the set of all points $(a_1, \dots, a_d)$ in $\Z^d_{>0}$ that satisfy $a_j \geq a_{j-1}$ for all $j \in [2,i]$, and $a_j \geq a_{j+1}$ for all $j \in [i+1,d-1]$.
	A \defi{$(d,i)$-Ferrers diagram} with respect to $p$ is a set $\F$ of $p$ points in $Z(d,i)$ such that $(a_1, \dots, a_d)$ is in $\F$ only if $(a'_1, \dots, a'_d)$ is also in $\F$, for all $(a'_1, \dots, a'_d) \in Z(d,i)$ that satisfies $a'_j \leq a_j$, $j \in [d]$.
\end{definition}

See \cref{fig:ferrers}(right) for an illustration of a $(d,i)$-Ferrers diagram.
Note that (2,1)-Ferrers diagrams are exactly the Ferrers diagrams.

We are now ready to characterize the co-signotopes in $\bar{\calS}_{p,i}(n,d)$.

\begin{lemma}
\label{lem:one_source_char}
    Let $n,d,p,i$ be integers such that $1 \leq p \leq n-d$ and $0 \leq i \leq d$.
    There is a one-to-one correspondence between a co-signotope $\tau$ in $\bar{\calS}_{p,i}(n,d)$ and a $(d,i)$-Ferrers diagram~$\F_{\tau}$ with respect to $p$.
    Further, this correspondence is described by $\F_{\tau} = \{g_{n,d,i}(B) \mid B \in \tau^{-1}(+)\}$. 
\end{lemma}
\begin{proof}
    Suppose we have a co-signotope $\tau \in \bar{\calS}_{p,i}(n,d)$, we show that for $\F_{\tau} := \{g_{n,d,i}(B) \mid B \in \tau^{-1}(+)\}$, the set $\F_{\tau}$ is a $(d,i)$-Ferrers diagram with respect to $p$.
    By \cref{ob:one_source_g}, the elements of $\F_{\tau}$ have positive integral indices, and $g_{n,d,i}$ is injective.
    Combined with the fact that $\tau$ has $p$ $+$-subsets, this implies that $\F_{\tau}$ has $p$ elements.

    We start by showing that $\F_{\tau} \subset Z(d,i)$, for $Z(d,i)$ as defined in \cref{def:gen_ferrers}.
    Consider a $+$-subset $B = (b_1, \dots, b_d)$.
    For $j \in [2,i]$, we have $b_j > b_{j-1}$ by our convention of a subset.
    Hence, $f_{n,d,i,j}(b_j) + j - 1 > f_{n,d,i,j-1}(b_{j-1}) + (j-1) - 1$, which is equivalent to $f_{n,d,i,j}(b_j) \geq f_{n,d,i,j-1}(b_{j-1})$.
    Similarly, for $j \in [i+1,n-1]$, we obtain $f_{n,d,i,j}(b_j) \geq f_{n,d,i,j+1}(b_{j+1})$.
    It follows that $\F_{\tau} \subset Z(d,i)$.
    Then \cref{lem:one_source_index} implies that $\F_{\tau}$ satisfies the condition of a $(d,i)$-Ferrers diagram with respect to~$p$ in \cref{def:gen_ferrers}.

    For the other direction, suppose we have a $(d,i)$-Ferrers diagram $\F$ with respect to $p$.
    We show that the sign function $\tau$ such that $\tau^{-1}(+) = \{g^{-1}_{n,d,i}(X) \mid X \in \F\}$ is a co-signotope in $\bar{\calS}_{p,i}(n,d)$.
    We first show that for every $X = (a_1, \dots, a_d) \in \F$, $g^{-1}_{n,d,i}(X)$ is a $d$-subset.
    Indeed, since $\F \subset Z(d,i)$, it is easy to see that for $j \in [2,i] \cup [i+2,n]$, $f^{-1}_{n,d,i,j}(a_j) > f^{-1}_{n,d,i,j-1}(a_{j-1})$.
    For $j=i+1$, note that if the sum of the indices of $X$ is $s$, then by \cref{def:gen_ferrers}, it is easy to see that for $k \in [d,s]$, there exists an element $Y$ in $\F$ whose sum of indices is exactly $k$.
    Combining this with the fact that there are $p$ elements in $\F$, we obtain that $s \leq p + d - 1$, and consequently, $a_i + a_{i+1} \leq p+1$.
    This implies
    \[
        f^{-1}_{n,d,i,i+1}(a_{i+1}) - f^{-1}_{n,d,i,i}(a_{i}) = (n - a_{i+1} - d + i + 2) - (a_i + i - 1) = (n-d + 3) - (a_{i+1} + a_i) \geq 2.
    \]
    This completes the proof that $g^{-1}_{n,d,i}(X)$ is a $d$-subset.
    
    Next, we show that $\tau$ is a co-signotope, i.e., for every $d$-subset $B$ and index $j$, the $(\tau,B,j)$-series has at most one sign change.
    This is implied by the condition of a $(d,i)$-Ferrers diagram in \cref{def:gen_ferrers}.
    
    Lastly, observe that for any element $X = (a_1, \dots, a_d)$ in $\F$, if it is not $(1,\dots, 1)$, then we can find an index $j$ such that the element $X'$ obtained from $X$ by replacing $a_j$ with $a_j-1$ is in $Z(d,i)$. 
    Hence, $X'$ is in $\F$.
    Repeating this process, we obtain a sequence of elements in $\F$ starting from $X$ and ending at $(1,\dots,1)$.
    This corresponds to a walk in the graph $G_{n,d}$ from a $+$-subset of $\tau$ to $S_i$.
    (Recall from \cref{ob:one_source_g} that $g^{-1}_{n,d,i}((1, \dots, 1)) = S_i$.)
    This implies that all $+$-subsets of $\tau$ are connected to $S_i$ in $G_n[\tau]$.
    Hence, there is only one $+$-component, and this $+$-component contains $S_i$.
    Since $g^{-1}_{n,d,i}$ is injective, the $+$-component has $p$ subsets.
    This completes the proof of the lemma.
\end{proof}

Observe that the set $\F_{\tau}$ in \cref{lem:one_source_char} depends only on $d$, $i$, and $p$, but not on $n$.
Hence, this lemma straighforwardly implies a bijection between $\bar{\calS}_{p,i}(n,d)$ and $\bar{\calS}_{p,i}(\tn,d)$ such that $\min\{n,\tn\} \geq p+d$, as follows.
\begin{corollary}
\label{cor:one_source_bijection}
    Let $n,\tn,d,p,i$ be integers such that $1 \leq p \leq \min\{n,\tn\}-d$ and $0 \leq i \leq d$.
    Let $h_{n,\tn,d,p,i}$ map a co-signotope $\tau$ in $\bar{\calS}_{p,i}(n,d)$ to a sign function $\tilde{\tau}: {[\tn] \choose d} \to \{+,-\}$ such that 
    \[ 
        \tilde{\tau}^{-1}(+) = \Big\{g^{-1}_{\tn,d,i}\big(g_{n,d,i}(B)\big) \mid B \in \tau^{-1}(+)\Big\}.
    \]
    Then $h_{n,\tn,d,p,i}$ is a bijection between $\bar{\calS}_{p,i}(n,d)$ and $\bar{\calS}_{p,i}(\tn,d)$ that preserves the single step relation.
\end{corollary}

\section{The $+$-component decomposition of a co-signotope}
\label{sec:decompositionPlusComp}

In the previous section, we studied co-signotopes with only one $+$-component. In this section, we go a step further and investigate how the $+$-subsets of a general co-signotope in $\bar{\calS}_{p}(n,d)$ can be decomposed into several $+$-components. 
The \defi{p-sequence} of a co-signotope $\tau$ in $\bar{\calS}_{p}(n,d)$ is the $(d+1)$-tuple $\big(p_0, \dots, p_d)$, where $p_i$ is the size of the $+$-component that contains $S_i$ for $i \in [0,d]$ (i.e., $p_i = |\C^{\tau}_i|$).
If we restrict the co-signotope $\tau$ such that we keep only one of its $+$-components as a $+$-component and set all remaining $+$-components to $-$, we still get a co-signotope.
More precisely, we define the \defi{$+$-component decomposition} of a co-signotope $\tau \in \bar{\calS}_{p}(n,d)$, denoted by $\Delta(\tau)$,
as the $(d+1)$-tuple $(\tau_0,\tau_1,\dots,\tau_d)$ such that for $i \in [0,d]$, each $\tau_i$ is a sign function with domain $\binnd$, where $\tau_i^{-1}(+)$ is exactly $\C^{\tau}_i$.
The following lemma shows that each $\tau_i$ is still a co-signotope. 

\begin{lemma}
\label{lem:decomp}
	Let $n,d,p$ be integers such that $0 \leq p \leq n-d$.
	For every co-signotope $\tau$ in $\bar{\calS}_{p}(n,d)$ with $\Delta(\tau) = (\tau_0,\tau_1,\dots,\tau_d)$, $\tau_i$ is a co-signotope in $\bar{\calS}_{|\C^{\tau}_i|,i}(n,d)$ for $i \in [0,d]$.
\end{lemma} 
\begin{proof}
	For every $d$-subset $B$ of $[n]$ and every index $j$, the $+$-subsets in the $(\tau,B,j)$-series form a path in $G_{n,d}$. Hence $\tau_i$ is a co-signotope in $\bar{\calS}_{|\C^{\tau}_i|,i}(n,d)$ for $i \in [0,d]$.
\end{proof}

Note that not every tuple $(p_0, \dots, p_d)$ such that $\sum^d_{i=0} p_i = p$ is the p-sequence of a co-signotope. 
There are some necessary conditions which turn out to be sufficient as well. For this we define a \defi{sparse composition} of an integer $p$ as a $(d+1)$-tuple $(c_0, \dots, c_d)$ of non-negative integers such that $\sum^n_{i=0} c_i = p$ and either $c_i=0$ or $c_{i-1}=0$ for all $i \in [d]$.

\begin{lemma}
\label{lem:decomp_reverse}
	Let $n,d,p$ be non-negative integers such that $p \leq n-d$.
    For every $\tau \in \bar{\calS}_{p}(n,d)$ its p-sequence $(p_0, \dots, p_d)$ is a sparse composition of $p$.
    Moreover, for every sparse composition $(c_0, \dots, c_d)$ of $p$ and every co-signotopes $\tau_0, 
	\dots, \tau_d$ such that $\tau_i \in \bar{\calS}_{c_i,i}(n,d)$, there exists a unique co-signotope $\tau$ in $\bar{\calS}_{p}(n,d)$ such that its p-sequence is $(c_0, \dots, c_d)$ and $\Delta(\tau) = (\tau_0,\tau_1,\dots,\tau_d)$.
\end{lemma}
\begin{proof}	
    To show the first part, note that by \cref{lem:exact_one_source}, we obtain $\sum^d_{i=0} p_i = p$.
	We argue that for all $i \in [p]$, either $p_i =0$ or $p_{i-1}=0$.
	Indeed, if this is not the case, 
    $S_i$ and $S_{i-1}$ are both $+$-subsets by definition. 
    Then the $(\tau,S_i,i)$-series starts with a $+$-subset $S_i$ and ends with a $+$-subset $S_{i-1}$.
	This implies that all subsets in $(\tau,B,i)$-series are $+$-subsets, and hence this series is flat, a contradiction to \cref{ob:no_flat}.
	Therefore, $(p_0, \dots, p_d)$ is a sparse composition of $p$.

	Now let $\tau$ be the sign function over $\binnd$ such that $\tau^{-1}(+) = \bigcup^d_{i=0} \tau_i^{-1}(+)$.
    By construction, $\tau$ has at most $p$ $+$-subsets.

	We start by arguing that $\tau$ is a co-signotope.
	Suppose that this is not the case.
	Then for some $d$-subset $B = (b_1, \dots, b_d)$ and an index $j$, the $(\tau,B,j)$-series has more than one sign change.
	By definition of $\tau$ and by the fact that $\tau_i$ is a co-signotope in $\bar{\calS}_{c_i,i}(n,d)$ for $i \in [0,d]$, we must have that for some $i, i' \in [0,d]$, the $(\tau_i,B,j)$-series is left-aligned and $(\tau_{i'},B,j)$-series is right-aligned.
	Because $(c_0, \dots, c_d)$ is a sparse composition, we have $|i-i'| \geq 2$.
	Let $B^{\alpha} = (b^{\alpha}_1, \dots, b^{\alpha}_d)$ and $B^{\omega} = (b^{\omega}_1, \dots, b^{\omega}_d)$ be the first and the last subsets of the $(B,j)$-series, respectively.
	We use the convention that $b_k = 0$ for $k < 1$ and $b_k = d+1$ for $k > d$.
	It is easy to see the following: 
    \begin{enumerate}[label=(\alph*)]
        \item If $j \geq i$, then $[b^{\alpha}_i, b^{\alpha}_{i+1}] \subseteq [b_{i-1},b_{i+1}]$;
        \item If $j < i$, then $[b^{\alpha}_i, b^{\alpha}_{i+1}] \subseteq [b_{i},b_{i+1}]$;
        \item If $j \leq i'+1$, then $[b^{\omega}_{i'}, b^{\omega}_{i'+1}] \subseteq [b_{i'},b_{i'+2}]$;
        \item If $j > i'+1$, then $[b^{\omega}_{i'}, b^{\omega}_{i'+1}] \subseteq [b_{i'},b_{i'+1}]$.
    \end{enumerate}

    From the above case distinction, we observe that when $j > i'+1$, the interiors of $[b^{\alpha}_i, b^{\alpha}_{i+1}]$ and $[b^{\omega}_{i'}, b^{\omega}_{i'+1}]$ are disjoint, since $|i'-i| \geq 2$.
    When $j \leq i' + 1$ and $i' \geq i + 2$, these interiors are also disjoint.
    When $j \leq i' + 1$ and $i' \leq i - 2$, then (b) and (c) are applicable, and the two interiors are also disjoint.

    Hence, we always have that the two interiors above are disjoint.
    Further, from (a)--(d) above, these interiors contain only integers from $[n] \setminus B \cup \{b_i, b_{i'+1}\}$.
    These two statements imply that $b^{\alpha}_{i+1} - b^{\alpha}_i + b^{\omega}_{i'+1} - b^{\omega}_{i'} \leq n-d + 4$.
    
	By \cref{lem:dist_to_source}, the shortest path between $S_i$ and $B^{\alpha}$ is $b^{\alpha}_i - b^{\alpha}_{i+1} + n -d + 1$, whereas the shortest path between $S_{i'}$ and $B^{\omega}$ is $b^{\omega}_{i'} - b^{\omega}_{i'+1} + n -d + 1$.
	Then the sum of the distances is 
	\[
	    b^{\alpha}_i - b^{\alpha}_{i+1} + n -d + 1 + b^{\omega}_{i'} - b^{\omega}_{i'+1} + n -d + 1 \geq
        2(n-d+1) - (n-d+4) \geq n-d-2 \geq p-2.
	\]
    This implies that $c_i + c_{i'} \geq p+1$, a contradiction to the fact that $(c_0, \dots, c_d)$ is a sparse composition of $p$.
	
	Next, by \cref{lem:exact_one_source}, each $+$-component of $\tau$ contains exactly one source subset.
    Hence, no subset is a $+$-subset of more than one co-signotope among $\tau_0, \dots, \tau_d$.
    Therefore, $\tau$ has exactly $p$ $+$-subsets, and each $+$-component of $\tau$ that contains $S_i$ corresponds exactly to the $+$-component of $\tau_i$.
    This also implies that $\tau$ is unique and $\Delta(\tau) = (\tau_0,\tau_1,\dots,\tau_d)$.
    It follows that $(c_0, \dots, c_d)$ is the p-sequence of $\tau$, completing the proof of the lemma.	
\end{proof}

\section{Proof of Theorem \ref{thm:bijection}}
\label{sec:main_proof}

Equipped with \cref{cor:one_source_bijection}, \cref{lem:decomp}, and \cref{lem:decomp_reverse}, we are now ready to prove \cref{thm:bijection}.

\begingroup
\def\thetheorem{\ref{thm:bijection}}
\begin{theorem}[Restated from \cref{subsec:co}]
    Let $n, \tilde{n}, d, p$ be integers such that $0\leq p \leq \min\{n,\tn\}-d$.
    Then there exists a bijection between $\bar{\calS}_{\leq p}(n,d)$ and $\bar{\calS}_{\leq p}(\tn,d)$ that preserves the relation~$\preceq_{\bar{S}}$.
\end{theorem}
\addtocounter{theorem}{-1}
\endgroup

\begin{proof}[Proof of \cref{thm:bijection}]
    Consider the following mapping $\zeta_{n,\tn,d,p}$ from a co-signotope $\tau \in \bar{\calS}_{\leq p}(n,d)$ to a co-signotope $\tilde{\tau} \in \bar{\calS}_{\leq p}(\tn,d)$.
    Let $(p_0, \dots, p_d)$ be the p-sequence of $\tau$.
    We first obtain the $+$-component decomposition $(\tau_0,\tau_1,\dots,\tau_d)$ of $\tau$.
    Next, we use the bijection $h_{n,\tn,d,p_i,i}$ in \cref{cor:one_source_bijection} to map each of $\tau_i \in \bar{\calS}_{p_i,i}(n,d)$ to $\tilde{\tau}_i \in \bar{\calS}_{p_i,i}(\tn,d)$.
    Then from the tuple $(\tilde{\tau}_0,\tilde{\tau}_1,\dots,\tilde{\tau}_d)$, we can combine their $+$-subsets to construct $\tilde{\tau}$.
    By \cref{lem:decomp_reverse}, $\tilde{\tau} \in \bar{\calS}_{p}(\tn,d)$. 

    By \cref{cor:one_source_bijection}, \cref{lem:decomp}, and~\cref{lem:decomp_reverse}, the mapping $\zeta_{n,\tn,d,p}$ is injective.
    Also by these corollary and lemmas, the mapping $\zeta_{\tn,n,d,p}$ maps $\tilde{\tau}$ back to $\tau$ and generally this mapping maps a signotope in $\bar{S}_p(\tilde{n},d)$ to $\bar{S}_p(n,d)$.
    This implies that $\zeta_{n,\tn,d,p}$ is a bijection between $\bar{S}_p(n,d)$ and $\bar{S}_p(\tilde{n},d)$.

    It remains to prove that $\zeta_{n,\tn,d,p}$ preserves the relation~$\preceq_{\bar{S}}$.
    For this, it is sufficient to prove that $\zeta_{n,\tn,d,p}$ preserves the single step relation.
    Let $\tau$ and $\rho$ be two $d$-co-signotopes in $\bar{\calS}_{\leq p}(n,d)$ such that $\tau^{-1}(+) \subset \rho^{-1}(+)$ and $|\rho^{-1}(+)| = |\tau^{-1}(+)| + 1 = k+1$ for some $k$.
    Let $(\tau_0, \dots, \tau_d) = \Delta(\tau)$ and $(\rho_0, \dots, \rho_d) = \Delta(\rho)$.
    This implies that $\Delta(\tau)$ and $\Delta(\rho)$ agree in all coordinates, except for some $i$-th coordinate, where $\tau_i$ and $\rho_i$ differ by a single step.
    Let $\tilde{\tau}$ and $\tilde{\rho}$ be the images of $\tau$ and $\rho$, respectively, in the bijection above.
    By \cref{cor:one_source_bijection}, it follows that $\Delta(\tilde{\tau})$ and $\Delta(\tilde{\rho})$ also agree in all coordinates, except for some $i$-th coordinate, where the two corresponding co-signotopes differ by a single step.
    This in turns implies that $\tilde{\tau}$ and $\tilde{\rho}$ differ by a single step.
    This completes the proof of the theorem.
\end{proof}

\paragraph*{A simple bijection.}
Although the mapping $\zeta_{n,\tn,d,p}$ in the proof above satisfies the conditions of \cref{thm:bijection}, its description is rather complicated.
We now provide a simpler description of the bijection.
In other words, we will prove that the mapping $\zeta_{n,\tn,d,p}$ is identical to the mapping $\phi_{n,\tn,d,p}$ as described below.

For a co-signotope $\tau \in \bar{\calS}_p(n,d)$, a $d$-subset $B = (b_1, \dots, b_d)$, we define $\gamma_{\tau, \tn}(B)$ as the tuple obtained from $B$ by replacing $b_j$ by $b_j + \tn -n$ if the $(\tau,B,j)$-series is right-aligned.

\begin{definition}
    Let $n, \tilde{n}, d,p$ be positive integers.
    Let $\phi_{n,\tn,d,p}$ map a co-signotope $\tau \in \bar{\calS}_{\leq p}(n,d)$ to a sign function $\tilde{\tau}: {\tn \choose d} \to \{+,-\}$ such that
    \[
        \tilde{\tau}^{-1}(+) = \Big\{\gamma_{\tau, \tn}(B) \mid B \in \tau^{-1}(+)\Big\}.
    \]
\end{definition}

\begin{lemma}
\label{lem:bijection}
    $\phi_{n,\tn,d,p}$ is a bijection as guaranteed by \cref{thm:bijection}.
\end{lemma}
\begin{proof}
    Consider the mapping $\zeta_{n,\tn,d,p}$ as described in the proof of \cref{thm:bijection}.
    We prove that it is identical to $\phi_{n,\tn,d,p}$.
    Note that the mapping $\zeta_{n,\tn,d,p}$ essentially maps each $+$-subset of $\tau$ to a $+$-subset of $\tilde{\tau}$.
    In the step of computing the $+$-component decomposition of $\tau$ and the step of combining $(\tilde{\tau}_0,\tilde{\tau}_1,\dots,\tilde{\tau}_d)$ into $\tilde{\tau}$, the $+$-subsets do not change, or rather, each $+$-subset is mapped to itself.
    In the middle step when we apply $h_{n,\tn,d,s_i,i}$, the mapping of the $+$-subsets is done through $g^{-1}_{\tn,d,i}\circ g_{n,d,i}$, by the definition of $h_{n,\tn,d,i}$ in \cref{cor:one_source_bijection}.
    By \cref{def:g}, we can describe $g^{-1}_{\tn,d,i}\circ g_{n,d,i}$ as follows:
    For a $d$-subset $B=(b_1, \dots, b_d)$ in the same $+$-component of $\tau$ as $S_i$, $g^{-1}_{\tn,d,i}( g_{n,d,i}(B))$ is obtained by replacing $b_j$ with $b_j$ for $j \in [i]$ and with $b_j+\tn-n $ for $j \in [i+1,d]$.
    Combining this description with \cref{lem:one_source_index}, we obtain that $g^{-1}_{\tn,d,i}( g_{n,d,i}(B)) = \gamma_{\tau,\tn}(B)$.
    It follows that $\zeta_{n,\tn,d,p} \equiv \phi_{n,\tn,d,p}$, as claimed.
\end{proof}

\section{Plus count}
\label{sec:count}

\cref{thm:bijection} implies that given $d$ and $p$, $\bar{S}_p(n,d)$ is a fixed number for all $n \geq p + d$.
This motivates the following definition.

\begin{definition}
    For two integers $d \geq 1$ and $p \geq 0$, the \defi{plus count} with respect to $d$ and $p$, denoted by $P_{d,p}$, is the size of $\bar{S}_p(n,d)$ for any $n \geq p + d$.
\end{definition}

From computer experiment, we obtain the values of $P_{d,p}$ for small $d$ and $p$, as presented in \cref{tab:plus_count}.

\begin{table}
    \centering
    \begin{tabular}{l | r r r r r r r r r r r}
        $d$ / $p$ & 0 & 1 & 2 & 3 & 4 & 5 & 6 & 7 & 8 & 9 & 10 \\
        \hline
        1 & 1 & 2 & 2 & 2 & 2 & 2 & 2 & 2 & 2 & 2 & 2 
        \\
        2 & 1 & 3 & 5 & 9 & 14 & 21 & 33 & 47 & 68 & 96 & 135 
        \\
        3 & 1 & 4 & 9 & 20 & 41 & 78 & 146 & 264 & 465 & 804 & 1368 
        \\
        4 & 1 & 5 & 14 & 36 & 86 & 192 & 413  & 857 & 1732 & 3422 & 6633 \\
        5 & 1 & 6 & 20 & 58 & 155 & 386 & 920 & 2110 & 4691 & 10176 & 21604 \\
        6 & 1 & 7 & 27 & 87 & 255 & 693 & 1790 & 4438 & 10636 & 24799 & 56485\\
        7 & 1 & 8 & 35 & 124 & 394 & 1154 & 3192 & 8444 & 21534 & 53292 & 128571 \\
    \end{tabular}
    \caption{The values of $P_{d,p}$ for small $d$ (vertical) and $p$ (horizontal)}
    \label{tab:plus_count}
\end{table}

We can provide a characterization of the plus counts based on the correspondences in Lemmas~\ref{lem:one_source_char}, \ref{lem:decomp}, and \ref{lem:decomp_reverse}.
Let $F_{d,i}(p)$ be the number of $(d,i)$-Ferrers diagrams with respect to $p$, and let $\D_{d+1,p}$ be the set of all sparse compositions of $p$.
Then the lemmas above give the following formula
\begin{equation}
\label{eq:plus_count}
    P_{d,p} = \sum_{(c_0,\dots,c_d) \in \D_{d+1,p}} \, \prod^d_{i=0} F_{d,i}(c_i).
\end{equation}

For the remainder of this section, we discuss a few values of $d$ and $p$.
We start with a few simple observations.

\begin{restatable}{lemma}{countone}
    It holds: 
    \begin{enumerate}[label=(\alph*)]
        \item \label{item:OneMinSign}
        For all $d \geq 1$: $P_{d,0} = 1$.
        \item \label{item:d=1}
        For all $p>0$: $P_{1,p} =2$/
        \item \label{item:p=1}
        For all $d \geq 1$: $P_{d,1} = d+1$.
    \end{enumerate}
\end{restatable}

\begin{proof}
    Firstly, \ref{item:OneMinSign} holds since for every choice of $n$ and $r$ there is exactly one co-signotope with no $+$-subsets, the unique minimal co-signotope with only $-$-subsets. 
    Secondly, when $d=1$, the definition of a co-signotope implies that removing any element from the sequence $(1, \dots, n)$ results in a sequence with at most one sign change.
    This can only be achieved in two ways: either $\sigma(i) = +$ for all $i \in [p]$ or $\sigma(i) = +$ for all $i \in [n-p+1,n]$.
    Hence, \ref{item:d=1} follows.
    Lastly, \ref{item:p=1} follows from the fact that there are $d+1$ source subsets as discussed in \cref{sec:plus_subsets}.
\end{proof}

Next, to prove the formulae for higher values of $p$, we show a few observations on $F_{d,i}$.

\begin{restatable}{lemma}{counttwo}
\label{lem:f_obs}
    For non-negative integers $d,i,p$ with $d \geq i$, it holds:
    \begin{enumerate}[label=(\alph*)]
        \item \label{item:f_sym} $F_{d,i}(p) = F_{d,d-i}(p)$.
        \item \label{item:f_decomp} 
            $F_{d,i}(p) = \sum^{p}_{c=1}F_{i,0}(c)F_{d-i,0}(p+1-c)$.
        \item For $d = 1$, $F_{d,i}(p) = 1$.
        \item \label{item:f_small}
            For $d > 1$, $F_{d,i}(0) = 1$, $F_{d,i}(1) = 1$, $F_{d,0}(2) = 1$, and $F_{d,0}(3) = 2$.
        \item
            For $0 < i < d$, $F_{d,i}(2) = 2$.
        \item 
            $F_{d,1}(3) = F_{d,d-1}(3) = 4$. For $1 < i < d-1$, $F_{d,i}(3) = 5$.
    \end{enumerate}
\end{restatable}    

\begin{proof}
    \ref{item:f_sym} follows from \cref{def:gen_ferrers}.
    Further, by this definition, the first $i$ coordinates and the last $d-i$ coordinates are independent from each other.
    Hence, \ref{item:f_decomp} holds.
    The next two items are easy to see.
    The last two items follow from \ref{item:f_decomp} and \ref{item:f_small}.
\end{proof}

Equipped with the lemma above, we can show the following.

\begin{restatable}{lemma}{countthree}
    For $d \geq 1$, $P_{d,2} = \frac{1}{2}d^2 + \frac{3}{2}d$ and $P_{d,3} = \frac{1}{6}d^3 + d^2 + \frac{17}{6}d - 2$.
\end{restatable}
\begin{proof}
    Let $c = (c_0, \dots, c_d)$ be a sparse composition of 3.
    Define $\alpha(c,d) := \prod^d_{i=0} F_{d,i}(c_i)$.
    We consider two cases.

    \textbf{Case 1.1:} \textit{For some $t$, $c_t = 2$}.
    There are $d+1$ such sparse compositions corresponding to the value of $t = 0, \dots, d$.
    By \cref{lem:f_obs}, $\alpha(c,d) = 1$ if $t = 0$ or $t = d$, and $\alpha(c,d) = 2$ otherwise. 

    \textbf{Case 1.2:} \textit{Exactly two of $c_0, \dots, c_d$ have value 1}.
    By \cref{lem:f_obs}, $\alpha(c,d) = 1$.
    Suppose $c_t$ and $c_{t'}$  with $t < t'$ are two coordinates with value 1.
    Then it is easy to see that $t$ and $t'-1$ are two distinct elements of $[0,d-1]$.
    Hence, there are $d \choose 2$ such sparse compositions.
    
    Applying (\ref{eq:plus_count}), we then have
    \[
        P_{d,2} = 2 \cdot 1 + (d-1) \cdot 2 + {d \choose 2} = \frac{1}{2}d^2 + \frac{3}{2}d.
    \]

    Next, let $c = (c_0, \dots, c_d)$ be a sparse composition of 3.
    We define $\alpha(c,d)$ as before.
    Consider the following cases.

    \textbf{Case 2.1:} \textit{For some $t$, $c_t = 3$}.
    There are $d+1$ such sparse compositions.
    By \cref{lem:f_obs}, $\alpha(c,d) = 2$ if $t=0$ or $t=d$; $\alpha(c,d) = 4$ if $t=1$ or $t=d-1$; and $\alpha(c,d) = 5$ otherwise.

    \textbf{Case 2.2:} \textit{For some $t$, $c_t = 2$}.
    If $t=0$ or $t=d$, there are $d-1$ such sparse compositions, and $\alpha(c,d) = 1$ by \cref{lem:f_obs}.
    For other value of $t$, there are $d-2$ such sparse compositions, and $\alpha(c,d) = 2$.

    \textbf{Case 2.3:} \textit{Exactly three of $c_0, \dots, c_d$ have value 1}.
    In this case, $\alpha(c,d) = 1$.
    Using similar argument as in Case 1.2, the number of such sparse compositions is $d-1 \choose 3$.

    Applying (\ref{eq:plus_count}), we then have
    \[
        P_{d,3} = 2 \cdot 2 + 2 \cdot 4 + (d-3) \cdot 5 + 2(d-1) \cdot 1 + (d-1)(d-2) \cdot 2 + {d-1 \choose 3} = \frac{1}{6}d^3 + d^2 + \frac{17}{6}d - 2.
    \]
\end{proof}

Lastly, recall that an integer partition of $n$ is the number of ways of writing $n$ as the (unordered) sum of positive integers.
Let $\pi(n)$ be the number of integer partitions of $n$.
Further, let $\delta(n)$ be the number of ways of writing $n$ as the (unordered) sum of distinct positive integers.

\begin{restatable}{lemma}{countfour}
    $P_{2,p} = \pi(p) + \sum^{p}_{i=0} \delta(i) \cdot \delta(p-i)$.
\end{restatable}

\begin{proof}
	The sparse 3-compositions of $p$ are either $(0,p,0)$ or of the form $(c,0,p-c)$ for some $c \in [0,p]$. 
    We now consider the values $F_{2,0}(p)$, $F_{2,1}(p)$, and $F_{2,2}(p)$.
    We start with the $(2,1)$-Ferrers diagrams.
    In this case, the set $Z(2,1)$ in \cref{def:gen_ferrers} is identical to $\Z^d_{>0}$.
    Hence, a $(2,1)$-Ferrers diagram of $p$ is essentially an integer partition of $p$.
    Therefore, $F_{2,1}(p) = \pi(p)$.
    For $(2,2)$-Ferrers diagrams, a point $(a_1,a_2) \in Z(2,2)$ must satisfy $a_2 \geq a_1$.
    Therefore a $(2,2)$-Ferrers diagram of $p$ is an integer partition of $p$ with distinct parts; see~\cref{fig:ferrers}(right).
    Therefore, $F_{2,2}(p) = \delta(p)$.
    Likewise, we also have $F_{2,0}(p) = \delta(p)$.
    Then the lemma follows from all of the above and (\ref{eq:plus_count}).
\end{proof}

Note that $\pi(n)$ and $\delta(n)$ are counted by the OEIS sequences A000041 and A000009~\cite{oeis}, respectively.
Further, the sum $\sum^{p}_{i=0} \delta(i) \cdot \delta(p-i)$ is counted by the OEIS sequence A022567.
Hence, $P_{2,p}$ is counted by the sum of the two sequences A000041 and A022567.

\section{Concluding remarks}
\label{sec:conclusions}

Firstly, we note that \cref{thm:bijection} is tight in the sense that we cannot relax the condition $p \leq \min\{n, \tn\} - d$.
As an example, consider the case when $d = n-1 = \tn - 2$ and $p = 2$.
In other words, we consider the possibility of a bijection between $\bar{S}_{\leq 2}(n,n-1)$ and $\bar{S}_{\leq 2}(n+1,n-1)$ or equivalently a bijection between $S_{\leq 2}(n,1)$ and $S_{\leq 2}(n+1,2)$.
As shown in \cref{app:tightness}, we have
\begin{align*}
    |S_{\leq 2}(n,1)| &= 1 + n + {n \choose 2},
    &|S_{\leq 2}(n+1,2)| &= 1 + n + {n \choose 2} + n - 1.
\end{align*}
Hence, there cannot be any bijection between $S_{\leq 2}(n,1)$ and $S_{\leq 2}(n+1,2)$.

An intuitive reason is that when $p > n-d$, \cref{ob:no_flat} no longer holds.
That is, there can be a flat series with only $+$-subsets.
For example, when $p = n-d+1$, there exists a co-signotope $\tau \in \bar{S}_{p}(n,d)$ such that its $+$-subsets form exactly the whole $(\tau,S_i,i)$-series for some $i \in [0,d]$.
In other words, its only $+$-component contains both $S_i$ and $S_{i+1}$.
However, for $\tn > n$ (i.e., $p \leq \tn-d)$, \cref{ob:no_flat} holds again, and it is unclear if $\tau$ should be mapped to a co-signotope $\tilde{\tau} \in \bar{S}_{p}(\tn,d)$ with a nonempty $C^{\tilde{\tau}}_i$ or
a nonempty $C^{\tilde{\tau}}_{i+1}$.

Secondly, from \cref{fig:1elementextension}, it seems that the bijection of \cref{thm:bijection} could be proved by the following simple geometric argument:
Suppose $n < \tn$; 
we remove $\tn-n$ hyperplanes that are completely in the negative side of the extending hyperplane of the one-element extension of $\mathbf{X}^{\tn,d}_c$.
However, this argument does not work in a high dimension.
For example, the positive side may consist of these two points: the intersection of the hyperplanes $H_1, \dots, H_d$ and that of $H_{\tn-d+1}, \dots, H_{\tn}$.
For $d$ and $\tn$ such that $d > \tn - d + 1$, there are then no hyperplanes completely on the negative side.

Lastly, it would be interesting to obtain understanding on more values of the plus counts, for example, their exact counts, generating function, or more relations.
As a starter, we conjecture that $P_{d,3} = P_{d-1,3} + P_{d-1,2} + P_{d-1,1} + 3$.

{
	\bibliography{references}

\begin{thebibliography}{10}

\bibitem{MR2122332}
George~E. Andrews and Kimmo Eriksson.
\newblock {\em Integer partitions}.
\newblock Cambridge University Press, Cambridge, 2004.

\bibitem{Balko19}
Martin Balko.
\newblock Ramsey numbers and monotone colorings.
\newblock {\em Journal of Combinatorial Theory, Series {A}}, 163:34--58, 2019.

\bibitem{BergoldFS23}
Helena Bergold, Stefan Felsner, and Manfred Scheucher.
\newblock {An Extension Theorem for Signotopes}.
\newblock In {\em 39th International Symposium on Computational Geometry (SoCG
  2023)}, volume 258 of {\em Leibniz International Proceedings in Informatics
  (LIPIcs)}, pages 17:1--17:14. Schloss Dagstuhl -- Leibniz-Zentrum f{\"u}r
  Informatik, 2023.

\bibitem{BjoernerLVWSZ1993}
Anders Bj{\"o}rner, Michel {Las Vergnas}, Bernd Sturmfels, Neil White, and
  G{\"u}nter~M. Ziegler.
\newblock {\em Oriented Matroids}, volume~46 of {\em Encyclopedia of
  Mathematics and its Applications}.
\newblock Cambridge University Press, 2 edition, 1999.

\bibitem{KuhnastDFS24}
Fernando {Cort{\'{e}}s K{\"{u}}hnast}, Justin Dallant, Stefan Felsner, and
  Manfred Scheucher.
\newblock An improved lower bound on the number of pseudoline arrangements.
\newblock In {\em 40th International Symposium on Computational Geometry, SoCG
  2024, June 11-14, 2024, Athens, Greece}, volume 293 of {\em LIPIcs}, pages
  43:1--43:18. Schloss Dagstuhl - Leibniz-Zentrum f{\"{u}}r Informatik, 2024.

\bibitem{FelsnerGoodman2017}
Stefan Felsner and Jacob~E. Goodman.
\newblock Pseudoline arrangements.
\newblock In Toth, O'Rourke, and Goodman, editors, {\em Handbook of Discrete
  and Computational Geometry}. CRC Press, third edition, 2017.

\bibitem{FelsnerWeil2000}
Stefan Felsner and Helmut Weil.
\newblock A theorem on higher {B}ruhat orders.
\newblock {\em Discrete \& Computational Geometry}, 23(1):121--127, 2000.

\bibitem{FelsnerWeil2001}
Stefan Felsner and Helmut Weil.
\newblock Sweeps, arrangements and signotopes.
\newblock {\em Discrete Applied Mathematics}, 109(1):67--94, 2001.

\bibitem{KaVo91}
Mikhail~M. Kapranov and Vladimir~A. Voevodsky.
\newblock Combinatorial-geometric aspects of polycategory theory: {Pasting}
  schemes and higher {Bruhat} orders (list of results).
\newblock {\em Cahiers de Topologie et Géométrie Différentielle
  Catégoriques}, 32:11--27, 1991.

\bibitem{Knuth1992}
Donald~E. Knuth.
\newblock {\em Axioms and Hulls}, volume 606 of {\em LNCS}.
\newblock Springer, 1992.

\bibitem{Levi1926}
F.~Levi.
\newblock Die {Teilung} der projektiven {Ebene} durch {Gerade} oder
  {Pseudogerade}.
\newblock {\em Berichte {\"u}ber die Verhandlungen der S{\"a}chsischen Akademie
  der Wissenschaften zu Leipzig, Mathematisch-Physische Klasse}, 78:256--267,
  1926.

\bibitem{manin1989arrangements}
Yu.~I. Manin and V.~V. Schechtman.
\newblock Arrangements of hyperplanes, higher braid groups and higher {B}ruhat
  orders.
\newblock {\em Advanced Studies in Pure Mathematics}, 17:289--308, 1989.
\newblock Algebraic Number Theory -- in honor of K. Iwasawa.

\bibitem{oeis}
{OEIS Foundation Inc.}
\newblock {The On-Line Encyclopedia of Integer Sequences}.
\newblock Published electronically at http://oeis.org.

\bibitem{Ziegler1993}
G{\"u}nter~M. Ziegler.
\newblock Higher {B}ruhat orders and cyclic hyperplane arrangements.
\newblock {\em Topology}, 32(2):259--279, 1993.

\end{thebibliography}
}

\appendix

\section{The counts of $S_{\leq 2}(n,1)$ and $S_{\leq 2}(n+1,2)$}
\label{app:tightness}

The signotopes of $S(n,1)$ are exactly the binary strings of size $n$, and the number of $+$-subsets in such a signotope is the number of ones in the corresponding binary string.
Hence, 
\[
    |S_{\leq 2}(n,1)| = 1 + n + {n \choose 2}.
\]

The signotopes of $S(n+1,2)$ are exactly the permutations of size $n+1$, and the number of $+$-subsets in such a signotope is the number of inversions in the corresponding permutation (i.e., the number of pairs $(i,j)$ such that $1 \leq i < j \leq n+1$, and the $i$-th symbol of the permutation is larger than its $j$-th symbol).
It is easy to see that there is one permutation with no inversion (i.e., the identity permutation $\text{id} = 1\dots (n+1)$) and $n$ permutations with one inversion (by transposing any two consecutive symbols of $\text{id}$).
The permutations with two inversions can be counted as follows.
For $i \in [n]$ and a permutation $\rho$, let $\sigma_i(\rho)$ be the permutation obtained from $\rho$ by transposing the $i$-th and $(i+1)$-th elements.
Then any permutation with two inversions can be written as $\sigma_i(\sigma_j(\text{id}))$ for some distinct $i$ and $j$.
If $|i-j| \geq 2$, then we have $\sigma_i(\sigma_j(\rho)) = \sigma_j(\sigma_i(\rho))$.
Otherwise, $\sigma_i(\sigma_j(\rho)) \neq \sigma_j(\sigma_i(\rho))$.
Hence, 
\[
    |S_{\leq 2}(n,1)| = 1 + n + {n \choose 2} + n-1.
\]
\end{document}